\documentclass[12pt]{article}

\usepackage{amsmath,amsthm,amsfonts,amssymb,amscd}
\usepackage{graphicx,geometry}

\allowdisplaybreaks[4]

\newtheorem {lemma}{Lemma}[section]
\newtheorem {theorem} {Theorem}[section]

\newtheorem {corollary}{Corollary}[section]

\newtheorem {claim}{Claim}[section]
\newtheorem {question}{Question}[section]
\numberwithin{equation}{section}
\allowdisplaybreaks [4]
\usepackage{fullpage}

\renewcommand{\le}{\leqslant}
\renewcommand{\ge}{\geqslant}

\begin{document}

\title{Spectral conditions implying the existence of doubly chorded cycles without or with constraints }

\author{Leyou Xu\footnote{Email: leyouxu@m.scnu.edu.cn}, Bo Zhou\footnote{Email: zhoubo@m.scnu.edu.cn} \\
School of Mathematical Sciences, South China Normal University\\
Guangzhou 510631, P.R. China}

\date{}
\maketitle

\begin{abstract}
What spectral conditions imply a graph contains a chorded cycle? This question was asked by R.J. Gould   in 2022. We answer two modified versions of Gould's question by
giving tight spectral conditions that imply the existence of a doubly chorded cycle and  a doubly chorded cycle with chords incident to one vertex, respectively. \\ \\
{\it Keywords:} doubly chorded cycle, spectral condition, spectral radius\\ \\
{\it Mathematics Subject Classifications:} 05C50, 05C38
\end{abstract}

\section{Introduction}

One of the classical questions in combinatorics is as follows: What bounds for the size imply an $n$-vertex graph contains a subgraph with a certain prescribed structure?
%In recent years,
Its spectral version is: What spectral conditions imply  an $n$-vertex  graph contains a subgraph with a certain prescribed structure?

We consider only simple graphs. Let $G$ be a graph with vertex set $V(G)$ and edge set $E(G)$. Let $\delta(G)$ denote the minimum degree of $G$.
The adjacency matrix $A(G)$ of $G$ is the matrix $(a_{uv})_{u,v\in V(G)}$, where
$a_{uv}=1$ if $u$ and $v$ are adjacent and $0$ otherwise. The largest eigenvalue of $A(G)$ is called the spectral radius of $G$, denoted by $\rho(G)$. The above question may be rephrased via the spectral radius as: Among $n$-vertex graphs that do not contain a subgraph with a certain prescribed structure, what is the maximum spectral radius and which graphs achieve this value? Such questions, originated from \cite{BE2} and called Brualdi--Solheid--Tur\'{a}n type problems in \cite{Ni}, received much attention in recent years, see, e.g. \cite{CDT,LP,Ni,WHM,ZLX,ZHL}.

Let $K_s$, $P_s$ and $C_s$ denote a complete graph on $s$ vertices, a path on $s$ vertices and a cycle on $s$ vertices. For positive integers $n_1,\dots, n_k$ with $k\ge 2$, let $K_{n_1,\dots, n_k}$ denote the complete $k$-partite graph with part sizes $n_1,\dots, n_k$.

Let $C$ be a cycle of a graph. We say an edge that joins two vertices of a cycle $C$ is a chord
of $C$ if the edge is not itself an edge of $C$.
 If $C$ has at least one chord, then it is called a chorded cycle,
and if $C$ has at least two chords, then it is called a doubly chorded cycle (DCC for short).
The study of cycles in graphs is well established. The existence of cycles with additional properties such as containing at least $s$ chords for some number $s\ge 1$ received much attention, see \cite{CFGH,CY,DMMS,Gou,QZ} and references therein. For $s\ge 1,2$, the cycles are just chorded cycles and doubly chorded cycles, respectively.

Given a graph $G$, P\'{o}sa \cite{Pos}  suggested the problem of finding degree conditions that imply the existence of a chorded cycle in a graph $G$. Czipser proved (see problem 10.2, p.~65 in \cite{Lov}) that $\delta(G)\ge 3$ suffices.
Gould   asked in \cite{Gou} the following question: What spectral conditions imply a graph contains a chorded cycle?
Zheng et al. \cite{ZHW} answered Gould's question by showing that any graph $G$ of order $n\ge 6$ with $\rho(G)\ge \rho(K_{2, n-2})$ contains a chorded cycle unless $G\cong K_{2,n-2}$. Different answers are provided via the signless Laplacian spectral radius \cite{XZ}.
Gould et al. \cite{GHH} observed  that it is only slightly more difficult to show that $\delta(G)\ge 3$ implies $G$ contains a DCC. So it is more natural to consider conditions implying the existence of DCCs \cite{GHH,Sv}. We try to answer the following two questions, which may be viewed as modified versions of
Gould's original question.

\begin{question} \label{LY}
What tight spectral conditions imply a graph on $n$ vertices contains a DCC?
\end{question}

The problem to determine the maximum size of an $n$-vertex graph
 that does not contain a cycle with many chords incident to a vertex received much attention. Through the work of Erd\H{o}s, Lewin (see \cite[p.~398, no.~12]{Bol}) and
 Bollob\'{a}s \cite[p.~398, no.~12]{Bol},  Jiang \cite{Ji} showed that for any $n$-vertex graph with $n\ge 9$  that does not contain a DCC with two chords incident to a vertex,  $|E(G)|\le 3n-9$, and it is attained if $G\cong K_{3, n-3}$.
 In \cite{XZ}, it was conjectured that there is an $n_0$ such that for any $n$-vertex graph $G$ that does not contain a DCC  with two chords incident to a vertex, $\rho(G)\le \rho(K_{3, n-3})$ with equality if and only if $G\cong K_{3,n-3}$ for $n\ge n_0$. This is part of the following question.

 We call a DCC with two chords incident to a vertex a DCC$_1$.

\begin{question} \label{LY2}
What tight spectral conditions imply a graph on $n$ vertices contains a DCC$_1$?
\end{question}

Let $G\cup H$ be the disjoint union of graphs $G$ and $H$. The disjoint union of $k$ copies of a graph $G$ is denoted by $kG$.
The join of disjoint graphs $G$ and $H$, denoted by $G\vee H$, is the graph obtained from $G\cup H$ by adding all possible edges between vertices in $G$ and vertices in $H$. %Let $K_1\vee P_4=K_1\vee P_4$.
Note that $K_1\vee P_4$ is the unique $5$-vertex DCC$_1$.
Zhao et al. \cite{ZHL} determined the $n$-vertex graphs that maximize the spectral radius among graphs
that do not contain any wheels $W_r:=P_1\vee C_r$ for $r\ge 3$.  A graph that does not contain copies of $K_1\vee P_4$ does not contain any wheels $W_r$ for $r\ge 4$ but it may contain $W_3$ (i.e., $K_4$).
A graph that does not contains any wheels may contain $K_1\vee P_4$.
There are spectral conditions for a graph with fixed order (size, respectively) which imply the existence of a copy of $K_1\vee P_4$ \cite{ZP} (\cite{YLP}, respectively). There is also a paper considering the spectral extremal graphs of graphs containing no copies of $K_1\vee P_{2k}$ with $k\ge 3$ for sufficiently large order \cite{YLY}.

%\begin{question} \label{LY3}
%What spectral conditions imply a graph on $n$ vertices contains a copy of $K_1\vee P_4$?
%\end{question}

Questions \ref{LY} and \ref{LY2} %and \ref{LY3}
are answered by the following Theorems \ref{double3},
and \ref{double2}, % and \ref{double1},
respectively.

\begin{theorem}\label{double3}
Suppose that $G$ is an $n$-vertex graph containing no DCCs, where $n\ge 3$. Then $\rho(G)\le \tfrac{1}{2}+\sqrt{2n-\tfrac{15}{4}}$ with equality if and only if $G\cong K_{1,1,n-2}$.
\end{theorem}

Let $F_1$ be the graph obtained from $K_2$ and $K_4$ by identifying a vertex.

\begin{theorem}\label{double2}
Suppose that $G$ is an $n$-vertex graph containing no DCC$_1$s. Then
\[
\rho(G)\le \begin{cases}
n-1&\mbox{ if }n\le 4\\
3.0861&\mbox{ if }n=5\\
\frac{1}{2}+\sqrt{2n-\frac{15}{4}}&\mbox{ if }6\le n\le 9\\
\sqrt{3(n-3)}&\mbox{ if }n\ge 10
\end{cases}
\]
with equality if and only if
\begin{equation}\label{FFF}
G\cong
\begin{cases}
K_n&\mbox{ if }n\le 4,\\
F_1&\mbox{ if }n=5,\\
K_{1,1,n-2}&\mbox{ if }6\le n\le 9,\\
K_{3,n-3}&\mbox{ if }n\ge 10.
\end{cases}
\end{equation}
\end{theorem}

\section{Preliminaries}

Let $G$ be a graph.
For $v\in V(G)$, we denote by $N_G(v)$ the neighborhood of $v$ in $G$, and $d_G(v)$ the degree of $v$ in $G$. Let $N_G[v]:=\{v\}\cup N_G(v)$ be the closed neighborhood of $v$. For simplicity, we use $d_u$ for $d_G(u)$ if there is no ambiguity. For $S\subseteq V(G)$, let $N_S(v)=N_G(v)\cap S$ and  $d_S(v)=|N_S(v)|$ whether $v\in S$ or not.
For $S,T\subset V(G)$, let $e(S,T)$ be the number of edges between $S$ and $T$. Particularly, if $S=T$, then $e(S)$ denotes the number of edges in $G[S]$ (the subgraph of $G$ induced by $S$). The set $S$ is independent if $G[S]$ has no edges.

For a nontrivial graph $G$ with $v\in V(G)$, $G-v$ denotes the graph obtained from $G$ by removing $v$ and its incident edges.

A vertex $v$ in a connected graph $G$ is a cut vertex if $G-v$ is disconnected, equivalently,
there exist vertices $u$ and $w$ distinct from $v$ such that there is at least one
path from $u$ to $w$ in $G$ and $v$ lies on every such path in $G$.

For a graph $G$ with two nonadjacent vertices $u$ and $v$, denote by $G+uv$ the graph obtained from $G$ by adding the edge $uv$. If $wz$ is an edge of $G$, then $G-wz$ is the graph obtained from $G$ be removing the edge $wz$.

The following lemma is an immediate consequence of the Perron-Frobenius theorem.

\begin{lemma}\label{addedges}
Let $G$ be a  graph and $u$ and $v$  two nonadjacent vertices of $G$. If $G+uv$ is connected,  then $\rho(G+uv)> \rho(G)$.
\end{lemma}

If $G$ is a connected graph, then $A(G)$ is irreducible, so by Perron-Frobenius theorem, there exists a unique unit positive eigenvector of $A(G)$ corresponding to $\rho(G)$, which we call the Perron vector of $G$.

\begin{lemma}\label{perron} \cite{BS}
Let $G$ be a connected graph with $\{u,v,w\}\subseteq V(G)$ such that $uw\notin E(G)$ and $vw\in E(G)$. If $\mathbf{x}$ is the Perron vector of $G$ with $x_u\ge x_v$, then $\rho(G-vw+uw)>\rho(G)$.
\end{lemma}

If $\mathbf{x}'$ is the Perron vector of $G-vw+uw$ in Lemma \ref{perron}, then  $x'_u>x'_v$ by Lemma \ref{perron}. So Lemma \ref{perron} has the equivalent form:
Let $G$ be a connected graph with $\{u,v\}\subset V(G)$ and  $N_G(v)\setminus N_G[u]\ne \emptyset$.
If $\mathbf{x}$ is the Perron vector of $G$ with $x_u\ge x_v$, then for any nonempty $N\subseteq  N_G(v)\setminus N_G[u]$, $\rho(G-\{vw: w\in N\}+\{uw: w\in N\})>\rho(G)$.

For a graph $G$ with $u,v\in V(G)$, the $(u,v)$-entry of $A(G)^2$ is equal to the number of walks of length two from $u$ to $v$. If $\mathbf{x}$ is the Perron vector of $G$, then $\mathbf{x}$ is an eigenvector of $A(G)^2$ corresponding to the eigenvalue $\rho(G)^2$, so
\begin{equation}\label{ruu}
\rho(G)^2 x_u=d_G(u)x_u+\sum_{v\in N_G(u)}d_{N_G(u)}(v)x_v+\sum_{v\in V(G)\setminus N_G[u]}d_{N_G(u)}(v)x_v.
\end{equation}

Suppose that $V(G)$ is partitioned as $V_1\cup \dots\cup V_m$. For $1\le i<j\le m$, set $A_{ij}$ to be the submatrix of $A(G)$ with rows corresponding to vertices in $V_i$ and columns corresponding to vertices in $V_j$. The quotient matrix of $A(G)$ with respect to the partition $V_1\cup \dots \cup V_m$ is denoted by $B=(b_{ij})$, where $b_{ij}=\tfrac{1}{|V_i|}\sum_{u\in V_i}\sum_{v\in V_j}a_{uv}$. If $A_{ij}$ has constant row sum, then we say $B$ is an equitable quotient matrix.
The following lemma is an immediate consequence of \cite[Lemma 2.3.1]{BH}.

\begin{lemma}\label{quo}
For a connected graph $G$, if $B$ is an equitable quotient matrix of $A(G)$, then the eigenvalues of $B$ are also eigenvalues of $A(G)$ and $\rho(G)$ is equal to the largest eigenvalue of $B$.
\end{lemma}

Let $G_1$ and $G_2$ be two graphs with $u\in V(G_1)$ and $v\in V(G_2)$. Denote by $G_1uvG_2$ the graph obtained by identifying the vertices $u$ and $v$, which we call this graph the coalescence of $G_1$ and $G_2$ at $u$ and $v$.
If $u=v$, then we denote by $G_1uG_2$ for simplicity.

\begin{lemma}\label{comm}\cite{CDG}
Let $G_1$ be a graph with $u\in V(G_1)$ and $G_2$ a connected graph with $v,w\in V(G_2)$. If $\rho(G_2-v)<\rho(G_2-w)$, then $\rho(G_1uvG_2)>\rho(G_1uwG_2)$.
\end{lemma}

\begin{lemma}\label{comm2}\cite{CDG}
Let $G$ be a connected graph with $u\in V(G)$, $H_1$ and $H_2$ be two graphs of the same order with $v_i\in V(H_i)$ for $i=1,2$. If $\rho(H_1)> \rho(H_2)$ and $\rho(H_1-v_1)\le \rho(H_2-v_2)$, then $\rho(Guv_1H_1)>\rho(Guv_2H_2)$.
\end{lemma}

\begin{lemma}\label{comp} \cite{EH}  We have
\[
\rho(K_{1,1,n-2})=\frac{1}{2}+\sqrt{2n-\frac{15}{4}} \mbox{ and } \rho(K_{3,n-3})=\sqrt{3(n-3)}\,.
\]
Moreover, $\rho(K_{3,n-3})>\rho(K_{1,1,n-2})$ if and only if $n\ge 10$.
\end{lemma}

Denote by $d_G(u,v)$ the distance between vertices $u$ and $v$ in a graph $G$.
Let $u$ and $v$ be two vertices of $G$ and $P$  a path from $u$ to $v$ in $G$. %Then $P^{-1}$ denotes the path from $v$ to $u$ obtained by reversing the vertices of $P$.
If $w$ and $z$ are two vertices on $P$ with $d_P(u,w)<d_P(u,z)$, then $P[w,z]$ denotes the segment of $P$ from $w$ to $z$.

If a component of a graph is $K_1$, $K_2$, or $K_{1,s}$ with $s\ge 2$, then we call it an isolated vertex, an isolated edge, or a star (in which the vertex of degree $s$ is its center). As usual, a copy of $K_3$ is called a triangle. A vertex of degree one is also called a pendant vertex.

We use $I_n$ to denote the identity matrix of order $n$.

\section{Proof of Theorem \ref{double3}}

%Given a graph $G$ of order $n$ and size $m$, P\'osa \cite{Po} showed  that if $m\ge 2n-3\ge 5$ then $G$ contains a chorded cycle.
%
%
%\begin{proposition} \label{RE}
%Let $G$ be a graph of order $n$ and size $m$ that does not contain a $2$-chorded cycle. Then $m\le 2n-3$ and the bound is sharp.
%\end{proposition}

\begin{lemma}\label{cdouble}%\cite{CXZ}
If $G$ is an $n$-vertex graph containing no DCCs, then $e(G)\le 2n-3$.
\end{lemma}

\begin{proof}
We prove by induction on $n$. If $n=3$, then $e(G)\le 3$, as desired. Suppose that $n\ge 4$ and the result is true for any graph of order less than $n$. As $G$ does not contain a DCC, $\delta(G)\le 2$. Let $v$ be the vertex in $G$ of minimum degree. By induction, $e(G-v)\le 2(n-1)-3$. So $e(G)=e(G-v)+d_G(v)\le 2(n-1)-3+2=2n-3$, as desired.
\end{proof}

Note that $K_{1,1,n-2}$ is an $n$-vertex graph containing no  DCCs which has $2n-3$ edges. So the bound  in the previous lemma is sharp.

\begin{proof}[Proof of Theorem \ref{double3}]
The case for $n=3,4$ can be checked easily. Suppose that $n\ge 5$.
Suppose that $G$ is a graph that maximizes the spectral radius among all graphs on $n$ vertices containing no DCCs. It suffices to show that $G\cong K_{1,1,n-2}$.

If $G$ is not connected with components $G_1,\dots, G_r$, then we obtain a graph $G'$ by choosing $v_i\in V(G_i)$ for $i=1,\dots, r$ and adding edges $v_1v_j$ for all $j=2,\dots, r$, which obviously does not contain a DCC, and we have by Lemma \ref{addedges} that $\rho(G')>\rho(G)$, which is a contradiction. Thus, $G$ is connected.

Let $\rho=\rho(G)$.
Denote by $\mathbf{x}$ the Perron vector of $G$. Let $u$ be a vertex with maximum entry in $\mathbf{x}$. Let $X=N_G(u)$ and $Y=V(G)\setminus N_G[u]$.
As $G$ does not contain a DCC, a component of $G[X]$ is an isolated vertex $K_1$, an isolated edge $K_2$, or a star $K_{1,r}$ for some $r\ge 2$. Let $X_0$ be the set of isolated  vertices in $G[X]$ and $X_1$ be the set of vertices of isolated edges in $G[X]$.  Let $X_2=X\setminus (X_0\cup X_1)$. Let $X_2'$ be the set of  vertices of degree at least two in $G[X]$ (i.e., the centers of the stars in $G[X]$).
Note that $\rho x_u=\sum_{v\in X}x_v$. From   \eqref{ruu}, we have
\begin{align*}
(\rho^2-\rho)x_u
&=d_ux_u+\sum_{v\in X}(d_X(v)-1)x_v+\sum_{v\in Y}d_X(v)x_v\\
%&=d_ux_u+\sum_{v\in X_2'}(d_X(v)-1)x_v+\sum_{v\in Y}d_X(v)x_v-\sum_{v\in X_0}x_v\\
&\le \left(d_u+\sum_{v\in X_2'}(d_X(v)-1)+\sum_{v\in Y}d_X(v)-\sum_{v\in X_0}\frac{x_v}{x_u}\right)x_u\\
&=\left(|X|+|X_2|-2|X_2'|+e(X,Y)-\sum_{v\in X_0}\frac{x_v}{x_u}\right)x_u,
\end{align*}
so
\[
 \rho^2-\rho\le |X|+|X_2|-2|X_2'|+e(X,Y)-\sum_{v\in X_0}\frac{x_v}{x_u}.
\]
As $K_{1,1,n-2}$ does not contain a DCC, we have by Lemma \ref{comp} that $\rho\ge \rho(K_{1,1,n-2})=\tfrac{1}{2}+\sqrt{2n-\tfrac{15}{4}}$, so
$\rho^2-\rho\ge 2n-4$.
Therefore, we have
\begin{equation}\label{3eq}
2n-4\le |X|+|X_2|-2|X_2'|+e(X,Y)-\sum_{v\in X_0}\frac{x_v}{x_u}.
\end{equation}

\begin{claim}\label{3cut}
There is no cut vertex in $V(G)\setminus \{u\}$.
\end{claim}
\begin{proof}
Suppose  to the contrary that there is a cut vertex $v\in V(G)\setminus \{u\}$. Then $G-v$ has at least two components, one of which does not contain $u$, which we denote by $H$. Let
\[
G'=G-\{vz:z\in N_H(v)\}+\{uz:z\in N_H(v)\}.
\]
Then  $u$ is a cut vertex of $G'$ and $G'[\{u\}\cup V(H)]\cong G[\{v\}\cup V(H)]$, so $G'$ does not contain a DCC.  However, we have by Lemma \ref{perron} that $\rho(G')>\rho$, a contradiction.
\end{proof}

\begin{claim}\label{3x1}
For each $vv'\in E(G[X_1])$, if $vw\in Y$ for some $w\in Y$, then either $N_X(w)\setminus\{v\}\subseteq X_0$ or $N_X(w)=\{v,v'\}$.
\end{claim}
\begin{proof}
We first claim that $N_X(w)\setminus\{v,v'\}\subseteq X_0$.  Otherwise, there is a vertex $v_1\in X\setminus (X_0\cup \{v,v'\})$  with $v_1w\in E(G)$. Then $v_1\in (X_1\cup X_2)\setminus \{v,v'\}$. Denote by $v_1'$ one neighbor of $v_1$ in $X$. Then $uv'vwv_1v_1'u$ is a cycle with chords $uv$ and $uv_1$, a contradiction.

If $v_0w\in E(G)$ for some $v_0\in X_0$ and $v'w\in E(G)$, then $uvv'wv_0u$ is a cycle with chords $uv'$ and $vw$, a contradiction. So $v'w\notin E(G)$ or $v_0w\notin E(G)$.
\end{proof}

Let $Y_1$ be the set of those vertices in $Y$ with at least one neighbor  in $X_1$.
By Claim \ref{3x1}, $d_{X_1}(w)=1,2$ for any $w\in Y_1$.
Let $Y_1^i=\{w\in Y_1:d_{X_1}(w)=i\}$ for $i=1,2$.
Then $Y_1=Y_1^1\cup Y_1^2$ and
\begin{equation}\label{3eqx1}
e(X_1,Y)=|Y_1^1|+2|Y_1^2|\mbox{ and }e(X_0,Y_1^2)=0.
\end{equation}

\begin{claim}\label{3x2}
For each $v\in X_2\setminus X_2'$, $N_Y(v)=\emptyset$.
\end{claim}
\begin{proof}
Suppose that $N_Y(v)\ne \emptyset$, say $w\in N_Y(v)$ for some $v\in X_2\setminus X_2'$. By Claim \ref{3cut}, there is a path $P$ between $w$ and $u$ which does not contain $v$. Let $v_0$ be the unique neighbor of $v$ in $X$ and $v'$ be a neighbor of $v_0$ different from $v'$ in $X$.

If $P$ does not contain $v_0$ or $v'$, then $uv'v_0vwPu$ is a cycle with chords $uv_0$ and  $uv$, a contradiction.
If $P$ contains $v_0$ but not $v'$ or $P$ contains both $v_0$ and $v'$ with $d_P(w,v_0)<d_P(w,v')$, then $uvwP[w,v_0]v_0v'u$ is a cycle with chords $uv_0$ and $vv_0$, a contradiction.
Similarly, if $P$ contains $v'$ but not $v_0$ or $P$ contains both $v'$ and $v_0$ with $d_P(w,v')<d_P(w,v_0)$, then $uvwP[w,v']v'v_0u$ is a cycle with chords $uv'$ and $vv_0$, also a contradiction.
\end{proof}

\begin{claim}\label{3x20}
For each $v\in X_2'$, if $vw\in E(G)$ for some $w\in Y$, then $N_X(w)\setminus\{v\}\subseteq X_0$.
\end{claim}
\begin{proof}
Suppose that there exists a vertex $v_0\in X_1\cup X_2\setminus\{v\}$ such that $v_0w\in E(G)$. Let $v'\in N_X(v)\setminus\{v_0\}$. By Claim \ref{3x2}, $vv_0\notin E(G)$.  Let $v_0'$ be a neighbor of $v_0$ in $X$.  Then $uv'vwv_0v_0'u$ is a cycle with chords $uv$ and $uv_0$, a contradiction.
\end{proof}

Let $Y_2$ be the set of those vertices in $Y$ with at least one neighbor  in $X_2$. Then by Claims \ref{3x2} and \ref{3x20},
\begin{equation}\label{3eqx2}
e(X_2,Y)=e(X_2',Y_2)=|Y_2|.
\end{equation}

By Claims \ref{3x1} and \ref{3x20}, we have $Y_1\cap Y_2=\emptyset$.

For $v\in X_1\cup X_2'$, let
\[
X_v=\{z\in X_0: N_Y(z)\cap N_Y(v)\ne \emptyset\}.
\]

\begin{claim}\label{3x12}
For each $v_0\in X_v$ with  $v\in X_1\cup X_2'$ and $w\in N_Y(v)\cap N_Y(v_0)$,
if $|N_Y(v_0)|\ge 2$, then $N_X(z)\subseteq \{v,v_0\}$ for each $z\in N_Y(v_0)\setminus\{w\}$. Moreover, if $|N_Y(v)|\ge 2$, then $|N_Y(v)\cap N_Y(v_0)|=1$ or $N_X(z)\subseteq \{v,v_0\}$ for each $z\in N_Y(v_0)$.
\end{claim}
\begin{proof}
The second part follows from the first part. We only show the first part.
Suppose that $N_X(z)\not\subseteq \{v,v_0\}$ for some $z\in N_Y(v_0)\setminus\{w\}$, say $v_1\in N_X(z)\setminus\{v,v_0\}$.
If $v\in X_1$, and $v_1$ is adjacent to $v$, then $uv_1zv_0wvu$ is a cycle with chords $uv_0$ and $vv_1$, a contradiction. Otherwise, denoting by $v'$ one neighbor of $v$ in $X$ (where $v'$ is neighbor of $v$ different from $v_1$ if $v\in X_2'$ and $v_1$ is a neighbor of $v$). Then $uv'vwv_0zv_1u$ is a cycle with chords $uv$ and $uv_0$, a contradiction.
\end{proof}

By Claim \ref{3x12},  we have $X_v\cap X_{v'}=\emptyset$ for any pair $\{v,v'\}\subseteq X_1\cup X_2'$.
By Claims \ref{3x1} and \ref{3x20} again, $e(X_v,Y_1\cup Y_2)=e(X_v,N_Y(v))$.

Let  $v\in X_1\cup X_2'$. Note that
\[
e(X_v,Y_1\cup Y_2)\le |X_v|+|N_Y(v)|.
\]
If $|N_Y(v)|=1$, this is obvious as  $e(X_v,Y_1\cup Y_2)=|X_v|$.  Suppose that
 $|N_Y(v)|\ge 2$. Let $X_v^{1}=\{v_0\in X_v:|N_Y(v)\cap N_Y(v_0)|=1\}$ and $X_v^2=X_v\setminus X_v^1$. Then  by Claim \ref{3x12}, $e(X_v^1,Y_1\cup Y_2)=|X_v^1|$ and $N_Y(v_0)\cap N_Y(v_0')=\emptyset$ for any pair $\{v_0,v_0'\}\subseteq X_v^2$.  So $e(X_v^2,Y_1\cup Y_2)=|\{w\in N_Y(v):v_0w\in E(G),v_0\in X_v^2\}|\le |N_Y(v)|$. Thus
\[
e(X_v,Y_1\cup Y_2)=e(X_v^1,Y_1\cup Y_2)+e(X_v^2,Y_1\cup Y_2)\le |X_v^1|+|N_Y(v)|\le |X_v|+|N_Y(v)|,
\]
as desired.

Let $X_0'=\cup_{v\in X_1\cup X_2'} X_v$. Then
\[
e(X_0',Y_1\cup Y_2)=\sum_{v\in X_1\cup X_2} e(X_v,Y_1\cup Y_2)\le\sum_{v\in X_1\cup X_2} \left(|X_v|+|N_Y(v)|\right)=|X_0'|+|Y_1^1|+|Y_2|.
\]

%For $v\in X_1\cup X_2'$,
Let
\[
Y_v=\{z\in Y\setminus (Y_1\cup Y_2): N_{X_v}(z)\ne \emptyset\}.
\]
%be the set of vertices with at least one neighbor in $X_v$.
By Claim \ref{3x12}, we have $e(X,Y_v)=e(X_v,Y_v)=|Y_v|$ and $Y_v\cap Y_{v'}=\emptyset$ for any $v'\in (X_1\cup X_2')\setminus\{v\}$.
Let $Y_0'=\cup_{v\in X_1\cup X_2'}Y_v$. Then $e(X_v,Y_0')=e(X_v,Y_v)=|Y_v|$ and
\[
e(X_0',Y_0')=\sum_{v\in X_1\cup X_2'} e(X_v,Y_0')=\sum_{v\in X_1\cup X_2'} |Y_v|=|Y_0'|.
\]
Note that $e(X_0',Y)=e(X_0',Y_0')+e(X_0',Y_1\cup Y_2)$. So
\begin{equation}\label{3eqx0}
e(X_0',Y)\le |X_0'|+|Y_0'|+|Y_1^1|+|Y_2|.
\end{equation}

Let $X_0^*=X_0\setminus X_0'$ and $Y_0=Y\setminus (Y_1\cup Y_2\cup Y_0')$. Then $e(X_0^*,Y)=e(X_0^*,Y_0)$ and $e(X_0',Y_0)=0$.
It then follows from  \eqref{3eqx1}--\eqref{3eqx0} that
\begin{align*}
e(X,Y)&=e(X_0^*,Y)+e(X_0',Y)+e(X_1,Y)+e(X_2',Y)\\
&\le e(X_0^*,Y)+|X_0'|+|Y_0'|+|Y_1^1|+|Y_2|+|Y_1^1|+2|Y_1^2|+|Y_2|\\
&= e(X_0^*,Y_0)+|X_0'|+|Y_0'|+2|Y_1|+2|Y_2|.
\end{align*}
Note that $n=1+|X|+|Y|$, $X=|X_0^*|+|X_0'|+|X_1|+|X_2|$ and $|Y|=|Y_0|+|Y_0'|+|Y_1|+|Y_2|$.
So, from   \eqref{3eq}, we have
\begin{equation}\label{3e}
|X_0^*|+|X_1|+2|X_2'|+2|Y_0|+|Y_0'|-2\le e(X_0^*,Y_0)-\sum_{v\in X_0}\frac{x_v}{x_u}.
\end{equation}

Suppose that $X_0^*\ne \emptyset$. Then $\sum_{v\in X_0}\frac{x_v}{x_u}>0$. From   \eqref{3e}, we have
\begin{equation}\label{bo}
e(X_0^*,Y_0)\ge |X_0^*|+2|Y_0|-1.
\end{equation}
Let $G_1=G[\{u\}\cup X_0^*\cup Y_0]$. As $G_1$ does not contain a DCC, we have
$e(G_1)\le 2|X_0^*|+2|Y_0|-1$ by Lemma \ref{cdouble}. So, from \eqref{bo}, we have
\[
2|X_0^*|+2|Y_0|-1\ge e(G_1)=|X_0^*|+e(X_0^*,Y_0)+e(Y_0)\ge 2|X_0^*|+2|Y_0|-1+e(Y_0),
\]
implying that  $e(Y_0)=0$. Thus \eqref{bo} is an equality, so from  \eqref{3e} again, we have $X_1=X_2'=\emptyset$,  so $G=G_1$.  As $G$ is bipartite, we have $\rho(G)\le \sqrt{e(G)}=\sqrt{2n-3}$, so
$\tfrac{1}{2}+\sqrt{2n-\tfrac{15}{4}}\le \sqrt{2n-3}$, a contradiction.
This shows that  $X_0^*=\emptyset$. So
 \eqref{3e} becomes
\begin{equation}\label{3ee}
|X_1|+2|X_2'|+2|Y_0|+|Y_0'|\le 2-\sum_{v\in X_0}\frac{x_v}{x_u}.
\end{equation}
If $X_1\cup X_2'=\emptyset$, then $X_0'=\emptyset$, so $X_0=\emptyset$, which shows that $G$ is trivial or disconnected, a contradiction.
So $X_1\cup X_2'\ne \emptyset$. From \eqref{3ee} we have
$(|X_1|, |X_2'|, |Y_0|, |Y_0'|)=(2,0,0,0)$, $(0,1,0,0)$ and $X_0=\emptyset$.

If $(|X_1|, |X_2'|, |Y_0|, |Y_0'|)=(0,1,0,0)$, then by Claims \ref{3cut} and \ref{3x2}, $Y_2=\emptyset$. So $G\cong K_{1,1,n-2}$, as desired.

If $(|X_1|, |X_2'|, |Y_0|, |Y_0'|)=(2,0,0,0)$, then we have from \eqref{3eq} and \eqref{3eqx1} that $Y_1^1=\emptyset$ and $V(G)=\{u\}\cup X_1\cup Y_1^2$.
We claim that $Y_1^2$ is independent, as otherwise, there is an edge $w_1w_2$ with $w_1,w_2\in Y$, which leads to a cycle $uv_1w_1w_2v_2u$ with chords $v_1v_2$ and $v_1w_2$, where $v_1,v_2\in X_1$, a contradiction. Therefore, $G\cong K_{1,1,n-2}$.

Combining the above cases, we conclude that $G\cong K_{1,1, n-2}$. This completes the proof.
\end{proof}

\section{Proof of Theorem \ref{double2}}

Let  $G$ be an $n$-vertex graph that does not contain a DCC$_1$. From a result in \cite{Ji}, we have  $e(G)\le 3n-9$ for $n\ge 9$. Note that $F_1$ is a graph on five vertices with no DCC$_1$s, and $e(F_1)=7>6=3n-9$.
In the following lemma, we show that   $e(G)\le 3n-9$ for $n\ge 6$. As in \cite{Ji},
we need a result due to Bondy \cite{Bo} (see also \cite[Theorem 4.11]{Bol}):
Let $C$ be a longest cycle (with length $c$) in an $n$-vertex graph $G$, then there are at most $\lfloor\tfrac{1}{2}c(n-c)\rfloor$ edges with at most one end vertex on $C$. We also need a result of Czipszer \cite{Cz} (see also \cite[p. 386]{Bol}) stating that if $\delta(G)\ge 4$ then $G$ contains a DCC$_1$.

\begin{lemma}\label{cdouble2}
Let $n\ge 6$.
If $G$ is an $n$-vertex graph that does not contain a DCC$_1$, then $e(G)\le 3n-9$ with equality when $G$ is bipartite if and only if $G\cong K_{3,n-3}$.
\end{lemma}
\begin{proof}
We prove the result by induction on $n$.

Suppose that $n=6$. Let $C$ be a longest cycle of $G$ and $c$ the length of $C$. By Bondy's result, $e(G-E(G[V(C)]))\le \lfloor\tfrac{1}{2}c(6-c)\rfloor$. As $G$ does not contain a DCC$_1$, each vertex in $G[V(C)]$ has degree at most three. Then $e(G[V (C)])\le \lfloor\frac{3}{2}c\rfloor$. Then $e(G)=e(G-E(G[V(C)])))+e(G[V (C)])\le  \lfloor\tfrac{1}{2}c(6-c)\rfloor+\lfloor\tfrac{3}{2}c\rfloor$. So, if $c=3,5,6$, then $e(G)\le 9$. Suppose that $c=4$. If $G[V(C)]\cong K_4$, then each vertex outside $C$ is adjacent to at most one vertex on $C$, and it is possible that the two vertices outside $C$ are adjacent to a vertex on $C$, so they may be adjacent, implying that $e(G)\le 6+3=9$. If $G[V(C)]\ncong K_4$, then $e(G[V(C)])\le 5$, so $e(G)\le \lfloor\tfrac{1}{2}c(6-c)\rfloor+5=9$. This shows that $e(G)\le 3n-9$ for $n=6$.

Moreover, suppose that $G$ is a bipartite graph with $e(G)=3n-9=9$.  Then $c=4,6$. If $c=4$, then $e(G)\le \lfloor\tfrac{1}{2}c(6-c)\rfloor+4=8<9$. If $c=6$, then
$e(G)= \lfloor\frac{3}{2}c\rfloor=9$ if and only if $G$ is  regular of degree three, equivalently, $G\cong K_{3,3}$.

Suppose that $n\ge 7$.  Let $v$ be a vertex of minimum degree in $G$. By induction assumption, $e(G-v)\le 3(n-1)-9$. By Czipser's result,  $\delta(G)\le 3$. So $e(G)= e(G-v)+\delta(G)\le 3(n-1)-9+3=3n-9$.

Next, suppose that $G$ is a bipartite graph with $e(G)=3n-9$. By the above argument,
$d_G(v)=3$ and $e(G-v)=3(n-1)-9$.
By induction assumption, $G-v\cong K_{3,n-4}$. If $n=7$, then as $G$ is bipartite, we have $G\cong K_{3,4}$. Suppose that $n\ge 8$. Let $(V_1, V_2)$ be the bipartition of $G-v$ with $|V_1|=3$ and $|V_2|=n-4$. Suppose that $G\not\cong K_{3,n-3}$. As $G$ is bipartite, $G$ does not contain a triangle, so $N_G(v)\subset V_2$. Let $V_1=\{u_1,u_2,u_3\}$ and $N_G(v)=\{v_1,v_2,v_3\}\subset \{v_1,\dots, v_4\}\subseteq V_2$. Then $vv_1u_1v_4u_3v_3u_2v_1v$ is a cycle with chords $v_1u_2$ and $v_1u_3$, a contradiction.
\end{proof}

For positive integers $n$ and $r$ with $n\ge 3r+4$ and  $r\ge 1$,
let $u$ be the vertex of degree $3r$ in $K_1\vee rK_3$ (any vertex if $r=1$), and $vw$ an edge of  $K_{3,n-3r-3}$, where the degree of $v$ is $n-3r-3$ and the degree of $w$ is three, and
let $H_{n,r}=(K_1\vee rK_3)uvK_{3,n-3r-3}$ and $H_{n,r}'=(K_1\vee rK_3)uwK_{3,n-3r-3}$.

\begin{lemma}\label{comp2}
For  $n\ge 3r+6$ and $r\ge 1$,
$\rho(H_{n,r})\ge \rho(H_{n,r}')$ with equality if and only if $n=3r+6$.
\end{lemma}
\begin{proof}
If $n=3r+6$, then $H_{n,r}\cong H_{n,r}'$. Suppose that $n\ge 3r+7$.
Then  $n-3r-3> 3$, so $\rho(K_{3,n-3r-3}-v)=\rho(K_{2,n-3r-3})=\sqrt{2(n-3r-3)}<\sqrt{3(n-3r-4)}=\rho(K_{3,n-3r-4})
=\rho(K_{3,n-3r-3}-w)$.
Thus the result follows by Lemma \ref{comm}.
\end{proof}

\begin{lemma}\label{comp3}
For positive integers $n\ge 3r+4$,
$\rho(K_{3,n-3})>\rho(H_{n,r})$.
\end{lemma}

\begin{proof}
Let $U_1=V(rK_3)$, $V_1$ be the set of vertices of degree three in $K_{3,n-3r-3}$ and $V_2=V(G)\setminus (U_1\cup V_1\cup \{u\})$.
With respect to the vertex partition $V(H_{n,r})=\{u\}\cup U_1\cup V_1\cup V_2$, $A(H_{n,r})$ has an equitable quotient matrix
\[
B=\begin{pmatrix}
0&3r&n-3r-3&0\\
1&2&0&0\\
1&0&0&2\\
0&0&n-3r-3&0
\end{pmatrix}.
\]
By Lemma \ref{quo}, $\rho(H_{n,r})$ is the largest root of $f(x)=0$, where
\[
f(x)=\det(xI_4-B)=x^4 - 2x^3 + (6r - 3n + 9)x^2 + (6n - 18r - 18)x + 6nr - 18r - 18r^2.
\]
It is evident that
$f''(x)\ge f''(\sqrt{3n-9})>0$ if $x\ge \sqrt{3n-9}$,
so it may be checked that  $f'(x)\ge f'(\sqrt{3n-9})>0$ if $x\ge \sqrt{3n-9}$.
As $n\ge 3r+4$, we have $4n-12-3r>3\sqrt{3n-9}$, so
\[
f(\sqrt{3(n-3)})=24rn-72r-18r^2-18r\sqrt{3n-9}=6r(4n-12-3r-3\sqrt{3n-9})>0,
\]
which shows that $\rho(H_{n,r})<\rho(K_{3,n-3})$.
\end{proof}

The following result follows from Lemmas \ref{comp2} and \ref{comp3}.

\begin{corollary}\label{comp4}
For  $n\ge 3r+6$ and $r\ge 1$, $\rho(K_{3,n-3})>\rho(H_{n,r}')$.
\end{corollary}

\begin{lemma}\label{cal}
The following results hold.\\
(i) For integer $n\ge 7$ with $n\equiv 1\pmod 3$, $\rho(K_1\vee \tfrac{n-1}{3}K_3)<\sqrt{3(n-3)}$\,;\\
(ii) For integer $n\ge 6$ with $n\equiv 0\pmod 3$, \[
\rho(K_1\vee (K_2\cup \tfrac{n-3}{3}K_3))<
\begin{cases}
\sqrt{3(n-3)} & \mbox{if $n\ge 9$},\\
\frac{1}{2}+\sqrt{2n-\frac{15}{4}} & \mbox{if $n=6$};
\end{cases}
\]
(iii) For integers $n\ge 8$ with $n\equiv 2\pmod 3$,
$\rho(K_1\vee (K_{1}\cup \tfrac{n-2}{3}K_3))<
\sqrt{3(n-3)}$\,.
\end{lemma}

\begin{proof}
From \cite[p.~19, Ex. 1.12]{BH}, we know that $\rho(K_1\vee \tfrac{n-1}{3}K_3)=1+\sqrt{n}$, so
 (i) follows.

% Let $G=K_1\vee \tfrac{n-1}{3}K_3$. Let $u$ be the vertex of degree $n-1$ in $G$ and $U_1=V(G)\setminus\{u\}$. With vertex partition $V(G)=\{u\}\cup U_1$, $A(G)$ has an equitable quotient matrix
%\[
%B=\begin{pmatrix}
%0&n-1\\
%1&2
%\end{pmatrix}.
%\]
%By Lemma \ref{quo}, $\rho(G)=\lambda(B)$, which is equal to the largest root $f(x)=x^2-2x-n+1=0$.

Let $H=K_1\vee (K_2\cup \tfrac{n-3}{3}K_3)$.
Let $u$ be the vertex of degree $n-1$ in $H$, $V_1$ be the set of vertices of degree two and $V_2=V(H)\setminus (V_1\cup \{u\})$. With respect to the vertex partition $V(H)=\{u\}\cup V_1\cup V_2$, $A(H)$ has an equitable quotient matrix
\[
B'=\begin{pmatrix}
0&2&n-3\\
1&1&0\\
1&0&2
\end{pmatrix}.
\]
By Lemma \ref{quo}, $\rho(H)$ is equal to the largest root of $g(x)=0$, where
\[
g(x)=\det(xI_3-B')=x^3 - 3x^2 - (n - 3)x + n + 1.
\]
As $g(x)$ is increasing in $[\sqrt{3n-9},\infty)$ and $g(\sqrt{3(n-3)})=2(n-3)\sqrt{3n-9}-2(4n-13)>0$ if $n\ge 9$, we have $\rho(H)<\sqrt{3(n-3)}$ if $n\ge 9$. The result for $n=6$ follows by an easy calculation. This proves (ii).

Let $G=K_1\vee (K_{1}\cup \tfrac{n-2}{3}K_3)$. Let $u$ be the vertex of degree $n-1$ and $v$ be the pendant vertex in $G$. Let $U_1=V(G)\setminus \{u,v\}$. With vertex partition $V(G)=\{u\}\cup \{v\}\cup V_1$, $A(G)$ has an equitable quotient matrix \[
B=\begin{pmatrix}
0&1&n-2\\
1&0&0\\
1&0&2
\end{pmatrix}.
\]
By Lemma \ref{quo}, $\rho(G)$ is equal to the largest root of $f(x)=0$, where
\[
f(x)=\det(xI_3-B)=x^3-2x^2-(n-1)x+2.
\]
As $f(x)$ is increasing in $[\sqrt{3n-9},\infty]$ and $f(\sqrt{3n-9})=(2n-8)\sqrt{3n-9}-2(3n-10)>0$.
%For $n=5$, the result follows from a direct calculation.
This proves (iii).
\end{proof}

\begin{lemma}\label{cal2}
For  $n\ge 5+t$ and $t\ge 2$ with $n-t\equiv 2\pmod 3$, $\rho(K_1\vee (K_{1,t}\cup \tfrac{n-t-2}{3}K_3))<\rho(K_{1,1,n-2})$.
\end{lemma}
\begin{proof}
Let $G=K_1\vee (K_{1,t}\cup \tfrac{n-t-2}{3}K_3)$, $\rho=\rho(G)$ and $\mathbf{x}$ be the Perron vector of $G$. Let $u$ and $v$ be the vertices of degree $n-1$ and $t+1$ in $G$, respectively. Let $V_1$ be the set of vertices of degree four in $G$ and $V_2$ be the set of remaining vertices.
By $A(G)\mathbf{x}=\rho\mathbf{x}$, for $i=1,2$, each vertex in $V_i$ has the same entry, denoted by $x_i$.
By $A(G)\mathbf{x}=\rho\mathbf{x}$ at $v$, $w\in V_1$ and $z\in V_2$, we have
\[
x_1=\frac{1}{\rho-2}x_u\mbox{ and }x_v=\frac{\rho+t}{\rho^2-t}x_u.
\]
As $\tfrac{1}{\rho-2}<\tfrac{\rho+t}{\rho^2-t}$, $x_1<x_v$.
Let \[
G'=G-\{wz\in E(G):w,z\in V_1\}+\{vw:w\in V_1\}.
\]
Then $G'\cong K_{1,1,n-2}$.
By Rayleigh's principle,
\[
\rho(G')\ge\mathbf{x}^\top A(G')\mathbf{x}=\mathbf{x}^\top A(G)\mathbf{x}-2\sum_{wz\in E(G[V_1])}x_wx_z+2\sum_{w\in V_1}x_vx_w>\rho,
\]
as desired.
\end{proof}

\begin{proof}[Proof of Theorem \ref{double2}]
It is trivial for $n=3,4$. Suppose in the following that $n\ge 5$.
Suppose that $G$ is a graph that maximizes the spectral radius among all graphs of order $n$ containing no DCC$_1$s. Similarly as in the proof of Theorem \ref{double3}, $G$ is connected.
Let $\rho=\rho(G)$ and let $\mathbf{x}$ be the Perron  vector of $G$ and $u$ a  vertex with maximum entry in $\mathbf{x}$. Let $X=N_G(u)$ and $Y=V(G)\setminus N_G[u]$. From   \eqref{ruu}, we have
\[
\rho^2x_u=d_ux_u+\sum_{v\in X}d_X(v)x_v+\sum_{v\in Y}d_X(v)x_v\le \left(|X|+2e(X)+e(X,Y)\right)x_u,
\]
so
\[
\rho^2\le |X|+2e(X)+e(X,Y).
\]
As $K_{3,n-3}$ has no DCC$_1$s, we have $\rho\ge \rho(K_{3,n-3})=\sqrt{3(n-3)}$, so we have
\begin{equation}\label{eqqq}
3(n-3)\le |X|+2e(X)+e(X,Y).
\end{equation}

By the same argument as in the proof of Claim \ref{3cut}, we have

\begin{claim}\label{cut}
There is no cut vertex in $V(G)\setminus\{u\}$.
\end{claim}
%\begin{proof}
%Suppose that there is a cut vertex $v\in V(G)\setminus \{u\}$. Then $G-v$ has at least two components, one of which does not contain $u$, which we denote by $H$. Let
%\[
%G'=G-\{vz:z\in N_H(v)\}+\{uz:z\in N_H(v)\}.
%\]
%Then  $u$ is a cut vertex of $G'$ and $G'[\{u\}\cup V(H)]\cong G[\{v\}\cup V(H)]$, so $G'\in \mathcal{G}_n$. However, we have by Lemma \ref{perron} that $\rho(G')>\rho$, a contradiction.
%\end{proof}

As $G$ does not contain a DCC$_1$, $G[X]$ is $P_4$-free,  any component  $G[X]$ is an isolated vertex $K_1$, an isolated edge $K_2$, a star $K_{1,r}$ for some $r\ge 2$, or a triangle $K_3$. Let $X_0$ be the set of isolated  vertices in $G[X]$, $X_1$ the set of vertices of isolated edges in $G[X]$, $X_2$ the set of vertices of stars  in $G[X]$, and  $X_3$ the set of vertices of triangles in $G[X]$. Then $X=X_0\cup X_1\cup X_2\cup X_3$.

\begin{claim}\label{x3}
For each $v\in X_3$, $N_Y(v)=\emptyset$.
\end{claim}
\begin{proof}
Suppose that $N_Y(v)\ne \emptyset$ for some $v\in X_3$. Assume that $w\in N_Y(v)$.
By Claim \ref{cut}, $v$ is not a cut vertex of $G$, so there is a path $P$ connecting $w$ and $u$ which does not contain $v$.  Let  $v_1$ and $v_2$ be two other vertices of the triangle of $G[X]$ containing $v$. If both $v_1$ and $v_2$ lie outside $P$, then $uv_1v_2vwPu$ is a cycle with chords $uv$ and $uv_2$, a contradiction. If one of $v_1$ or $v_2$, say $v_1$, lies on $P$, where we assume that $d_P(v_1,w)<d_P(v_2,w)$ if both $v_1$ and $v_2$ lie on $P$,  then $uv_2vwP[w,v_1]v_1u$ %$uv_1P^{-1}[v_1, w]wvv_2u$
is a cycle with chords $vu$ and $vv_1$, also a contradiction.
\end{proof}

Let $X_2'$ be the set of the centers of the stars in $G[X_2]$.
Let $X_2^*$ be the set of pendant vertices of the components $K_{1,2}$ of $G[X_2]$

\begin{claim}\label{x2}
For each $v\in X_2^*$, if $vw\in Y$ for some $w\in Y$, then $N_X(w)\subseteq \{v,v'\}$, where $v'$ is the other pendant vertex  in the component of $G[X_2]$ containing $v$.
\end{claim}
\begin{proof}
Suppose  that $wv_1\in E(G)$ for some $v_1\in X\setminus \{v,v'\}$.
If $vv_1\in E(G)$, then $v'v_1\in E(G)$, so $uv'v_1wvu$ is a cycle with chords $uv_1$ and $vv_1$, a contradiction. So $vv_1\notin E(G)$.
Let $v_0$ be the unique neighbor of $v$ in $G[X]$. Then  $uv'v_0vwv_1u$ is a cycle with chords $uv_0$ and $uv$, also a contradiction.
\end{proof}

\begin{claim}\label{x21}
For each $v\in X_2\setminus (X_2'\cup X_2^*)$, $N_Y(v)=\emptyset$.
\end{claim}
\begin{proof}
Suppose that  $N_Y(v)\ne \emptyset$ for some  $v\in X_2\setminus (X_2'\cup X_2^*)$, say
that  $w\in N_Y(v)$. By Claim \ref{cut}, $v$ is not a cut vertex of $G$, so there is a path $P$ from $w$ to $u$ which does not contain $v$.
Let $v_0$ be the neighbor of $v$ in $X$ and  $v_1, \dots, v_t$ all the neighbors of $v_0$ in $X$ with $v=v_t$. If all $v_0,v_1, \dots, v_{t-1}$ lie outside $P$, then $uv_1v_0vwPu$ is a cycle with chords $uv_0$ and $uv$, a contradiction.
Suppose that one of $v_1,\dots, v_{t-1}$, say $v_1$,  lies on $P$.
If
$v_0$ lies on $P$ and  $d_P(w,v_0)<d_P(w,v_1)$, then $uvwP[w,v_0]v_0v_1u$ is a cycle with chords $uv_0$ and $vv_0$, a contradiction. Otherwise, we have either
$v_0$ lies outside $P$, or $v_0$ lies on $P$ and  $d_P(w,v_0)>d_P(w,v_1)$. As $t\ge 3$, $uv_2v_0vwP[w,v_1]v_1u$ is a cycle with chords
 $v_0u$ and $v_0v_1$, also a contradiction.
\end{proof}

By Claim \ref{x2}, for any $v\in X_2^*$ and any $w\in Y$, if $vw\in E(G)$, then $d_X(w)\le 2$. Let  $Y_2$ be the set of those vertices in $Y$ with  at least one neighbor in $X_2^*$. Then
$e(X_2^*,Y_2)\le  2|Y_2|$. By Claim \ref{x21}, $e(X_2\setminus (X_2'\cup X_2^*), Y)=0$, so
\begin{equation}\label{ex2y}
e(X_2\setminus X_2',Y)=e(X_2^*,Y_2)\le  2|Y_2|.
\end{equation}

By a similar argument as in the proof of Claim \ref{3x20}, we have

\begin{claim}\label{x20}
For each $v\in X_2'$, if $vw\in Y$ for some $w\in Y$, then $N_X(w)\setminus\{v\}\subseteq X_0$.
\end{claim}
%\begin{proof}
%Suppose that $wv_1\in E(G)$ for some $v_1\in X\setminus (X_0\cup \{v\})$. By Claims \ref{x3}--\ref{x21}, $v_1\in X_1\cup X_2'$. Let $v'$ be a neighbor of $v$ in $X$, and $v_1'$ a neighbor of $v_1$ in $X$. Then $uv'vwv_1v_1'u$ is a cycle with chords $uv$ and $uv_1$, a contradiction.
%\end{proof}

Denote by $Y_2'\subseteq Y$ the set of vertices with at least one neighbor  in $X_2'$. Then by Claim \ref{x20},
\begin{equation}\label{ex20y}
e(X_2',Y)=e(X_2',Y_2')=|Y_2'|.
\end{equation}

\begin{claim}\label{x1}
For each $vv'\in E(G[X_1])$, if $vw\in Y$ for some $w\in Y$, then $N_X(w)\setminus \{v,v'\}\subseteq X_0$.
\end{claim}

\begin{proof}
Suppose that $wv_1\in E(G)$ for some $v_1\in X\setminus (X_0\cup  \{v,v'\})$.  Denote by $v_1'$ a neighbor of $v_1$ in $X$. Then $uv'vwv_1v_1'u$ is a cycle with chords $uv$ and $uv_1$, a  contradiction.
\end{proof}

Let $Y_1\subseteq Y$ be the set of those vertices with at least one neighbors in $X_1$.
By Claims \ref{x2}, \ref{x20} and \ref{x1}, $S\cap T=\emptyset$ for each pair of $\{S,T\}\subset \{Y_1,Y_2,Y_2'\}$.

For $v\in X_1\cup X_2'$, let $X_v=\{z\in X_0: N_Y(z)\cap N_Y(v)\ne \emptyset\}$.

By the same argument as in the proof of Claim \ref{3x12}, we have

\begin{claim}\label{x12}
For each $v_0\in X_v$, if $|N_Y(v_0)|\ge 2$, then $N_X(z)\subseteq \{v,v_0\}$ for each $z\in N_Y(v_0)\setminus\{w\}$, where $w$ is a common neighbor of $v$ and $v_0$. Moreover, if $|N_Y(v)|\ge 2$, then $|N_Y(v)\cap N_Y(v_0)|=1$ or $N_X(z)\subseteq \{v,v_0\}$ for each $z\in N_Y(v_0)$.
\end{claim}

By Claim \ref{x12}, $X_v\cap X_{v'}=\emptyset$ for any pair $\{v,v'\}\subseteq X_1\cup X_2'$ with $vv'\not\in E(G)$ if $v,v'\in X_1$. Then for each $v\in X_2'$, by Claim \ref{x20}, $e(X_v,Y_1\cup Y_2')=e(X_v,N_Y(v))$ and so
by the same argument as in the previous section,
\begin{equation}\label{eqn}
e(X_v,Y_1\cup Y_2')\le |X_v|+|N_Y(v)|.
\end{equation}

\begin{claim}\label{x11}
If $vv'\in E(X_1)$, then $X_v\subseteq X_{v'}$ or $X_{v'}\subseteq X_v$.
\end{claim}
\begin{proof}
Suppose to the contrary that $v_0\in X_v\setminus X_v'$ and $v_1\in X_v'\setminus X_v$. By the definition of $X_v$, there exist vertices $w$ and $w'$ such that $vw,v_0w,v'w',v_1w'\in E(G)$. Then $uv_0wvv'w'v_1u$ is a cycle with chords $uv$ and $uv'$, a contradiction. So $X_v\setminus X_v'=\emptyset$ or $X_v'\setminus X_v=\emptyset$.
\end{proof}

For an edge $vv'\in G[X_1]$, assume that $X_{v'}\subseteq X_v$ by Claim \ref{x11}. For $vv'\in E(X_1)$, let $Y_{vv'}=\{w\in Y_1:vw,v'w\in E(G)\}$.
\begin{claim}\label{eqx1}
 For $vv'\in G[X_1]$,
$e(\{v,v'\}\cup X_v\cup X_{v'},Y_1)\le 2|N_Y(v)\cup N_Y(v')|+|X_v|$.
\end{claim}
\begin{proof}
If $X_v=\emptyset$, then the result follows by Claim \ref{x1}. Suppose next that $X_v\ne\emptyset$, say $v_0\in X_v$. Let $w\in N_Y(v_0)$. Assume that $vw\in E(G)$.

We first claim that $Y_{vv'}\setminus\{w\}=\emptyset$. Otherwise, there is some vertex $z\in Y_{vv'}$ with $z\ne w$. Then $uv_0wvzv'u$ is a cycle with chords $uv$ and $vv'$, a contradiction. So either $Y_{vv'}=\{w\}$ or $Y_{vv'}=\emptyset$.

If $Y_{vv'}=\{w\}$, then $e(\{v,v'\},Y_1)=|N_Y(v)\cup N_Y(v')|+1\le 2|N_Y(v)\cup N_Y(v')|$ and for each $z\in X_v$, $z$ is adjacent to exactly one vertex $w$ in $Y$, so $e(X_v,Y_1)=|X_v|$. Therefore, $e(\{v,v'\}\cup X_v\cup X_{v'},Y_1)\le 2|N_Y(v)\cup N_Y(v')|+|X_v|$.

If $Y_{vv'}=\emptyset$, then $e(\{v,v'\},Y_1)=|N_Y(v)\cup N_Y(v')|$. By  Eq.~\eqref{eqn}, $e(X_v,Y_1)\le |X_v|+|N_Y\cup N_Y(v')|$. Therefore, $e(\{v,v'\}\cup X_v\cup X_{v'},Y_1)\le 2|N_Y(v)\cup N_Y(v')|+|X_v|$.
\end{proof}

Let $X_0'\subseteq X_0$ be the set of those vertices with at least one neighbor in $Y_1\cup Y_2'$. Then $X_0'=\cup_{v\in X_1\cup X_2'}X_v$.
For $v\in X_1\cup X_2'$, let $Y_v\subseteq Y\setminus (Y_1\cup Y_2\cup Y_2')$ be the set of those vertices with at least one neighbor  in $X_v$. Let $Y_0'=\cup_{v\in X_1\cup X_2'}Y_v$. By Claim \ref{x12},
\[
e(X_v,Y_0')=e(X_v,Y_v)=|Y_v|.
\]
Now, from \eqref{eqn},
\[
e(X_v, Y)=e(X_v, Y_1\cup Y_2'\cup Y_0')=e(X_v, Y_1\cup Y_2')+e(X_v, Y_0')=|X_v|+|N_Y(v)|+ |Y_v|.
\]
By Claim \ref{eqx1} together with  \eqref{ex20y},
\begin{align*}
&\quad e(X_0'\cup X_1\cup X_2',Y)\\
&=\sum_{vv'\in E(X_1)}\left(e(\{v,v'\}\cup X_v\cup X_{v'},Y_1)+e(X_v\cup X_{v'},Y_0')\right)+\sum_{v\in X_2'}\left(e(v,Y)+e(X_v,Y)\right)\\
&=\sum_{vv'\in E(X_1)}\left(e(\{v,v'\}\cup X_v\cup X_{v'},Y_1)+e(X_v,Y_0')\right)+
+e(X_2',Y)+
\sum_{v\in X_2'}e(X_v,Y)\\
&\le \sum_{vv'\in E(X_1)} (2|N_Y(v)\cup N_Y(v')|+|X_v|+|Y_v|)+|Y_2'|+\sum_{v\in X_2'}(|X_v|+|N_Y(v)|+|Y_v|)\\
&=2|Y_1|+\sum_{vv'\in E(X_1)}(|X_v|+|Y_v|)+2|Y_2'|+\sum_{v\in X_2'}(|X_v|+|Y_v|),
\end{align*}
i.e.,
\begin{equation}\label{ex0y12}
e(X_0'\cup X_1\cup X_2',Y)\le  |X_0'|+|Y_0'|+2|Y_1|+2|Y_2'|.
\end{equation}

Let $X_0^*=X_0\setminus X_0'$ and $Y_0=Y\setminus (Y_1\cup Y_2'\cup Y_2\cup Y_0')$. Then $e(X_0^*,Y)=e(X_0^*,Y_0)$.
By  \eqref{ex2y} and \eqref{ex0y12},
\begin{align*}
e(X,Y)&=e(X_0^*,Y)+e(X_0'\cup X_1\cup X_2',Y)+e(X_2\setminus X_2',Y)\\
& \le e(X_0^*,Y_0)+|X_0'|+|Y_0'|+2|Y_1|+2|Y_2'|+2|Y_2|.
\end{align*}
Note that
\[
e(X)=\frac{1}{2}|X_1|+|X_2|-|X_2'|+|X_3|
\]
and
\[
n=1+|X_0^*|+|X_0'|+|X_1|+|X_2|+|X_3|+|Y_0|+|Y_0'|+|Y_1|+|Y_2'|+|Y_2|.
\]
From  \eqref{eqqq}, we have
\begin{equation}\label{eq1}
e(X_0^*,Y_0)\ge 2|X_0^*|+|X_0'|+|X_1|+2|X_2'|+3|Y_0|+2|Y_0'|+|Y_1|+|Y_2'|+|Y_2|-6,
\end{equation}
so
\begin{equation}\label{ex0y0}
e(X_0^*,Y_0)\ge 2|X_0^*|+3|Y_0|-6.
\end{equation}

\noindent
{\bf Case 1.} $|X_0^*|+|Y_0|\ge 5$.

Let  $G_1=G[\{u\}\cup X_0^*\cup Y_0]$. From  \eqref{ex0y0}, we have
\[
e(G_1)\ge |X_0^*|+e(X_0^*)+e(X_0^*,Y_0)+e(Y_0)\ge 3|X_0^*|+3|Y_0|-6+e(Y_0).
\]
As $G_1$ does not contain a DCC$_1$, we have $e(G_1)\le 3|X_0^*|+3|Y_0|-6$, so
\[
3|X_0^*|+3|Y_0|-6+e(Y_0)\le e(G_1)\le 3|X_0^*|+3|Y_0|-6,
\]
implying that $e(Y_0)=0$ and $e(G_1)=3|X_0^*|+3|Y_0|-6$.

As $e(Y_0)=0$, $G_1$ is a bipartite graph with bipartition $(X_0^*, \{u\}\cup Y_0)$.
As $e(G_1)=3|X_0^*|+3|Y_0|-6$, we have by Lemma \ref{cdouble2} that $G_1\cong K_{3, |X_0^*|+|Y_0|-2}$. So
$|X_0|=3$ or $|Y_0|=2$.

As $e(G_1)=3|X_0^*|+3|Y_0|-6$, we know from the above argument that
\eqref{ex0y0} is an equality. Then, from \eqref{eq1},  we have $X_0'=X_1=X_2'=Y_0'=Y_2'=Y_2=\emptyset$, so $X=X_0^*\cup X_3$ and $Y=Y_0$.

Suppose  $X_3\ne \emptyset$. Let $r$ be the number of triangles in $G[X]$. Then $|X_3|=3r$. If $|Y_0|=2$, then $G\cong H_{n,r}$, and  if $|X_0|=3$, then $G\cong H_{n,r}'$.  By Lemma \ref{comp3} and Corollary \ref{comp4},
we have $\rho<\rho(K_{3,n-3})$, a contradiction. It follows that $X_3=\emptyset$, so $n\ge 6$ and $G\cong K_{3,n-3}$.

\noindent
{\bf Case 2.} $|X_0^*|+|Y_0|\le 4$ and $X_0^*\ne \emptyset$.

As $e(X_0^*,Y_0)\le |X_0^*||Y_0|$, we have from \eqref{ex0y0} that
$(|X_0^*|-3)(|Y_0|-2)\ge 0$, so $|X_0^*|\le 3$ and $|Y_0|\le 2$.
If $|X_0^*|=1$ and $|Y_0|\ge 1$, then we have from \eqref{eq1} that  $|X_0'|+|X_1|+2|X_2|+2|Y_0'|+|Y_1|+|Y_2|+|Y_2'|\le 4-2|Y_0|\le 2$, so $Y_0'=Y_1=Y_2=Y_2'=\emptyset$,  implying  $u$ and any vertex in $Y_0$ belong to different components of $G-X_0^*$, that is,
the vertex of $X_0^*$ is a cut vertex, which contradicts Claim \ref{cut}. So if $|X_0^*|=1$, then $Y_0=\emptyset$ and there are three possibilities as below:
\begin{enumerate}
\item[(a)] $|X_0^*|=2$ and $|Y_0|=2$;

\item[(b)] $|X_0^*|=2,3$ and $|Y_0|=1$;
\item[(c)] $|X_0^*|=1,2,3$ and $Y_0=\emptyset$.
\end{enumerate}

Suppose first that (a) holds. From   \eqref{ex0y0}, we have  $e(X_0^*,Y_0)=4$.  Now, from  \eqref{eq1}, we have $X_0'=X_1=X_2'=Y_0'=Y_1=Y_2=\emptyset$, so $X=X_0^*$ and $Y=Y_0$.  If $X_3\ne \emptyset$, then $G\cong H_{n,\frac{|X_3|}{3}}$,  so we have by Lemma \ref{comp3} that $\rho<\rho(K_{3,n-3})$, a contradiction. It follows that $X_3=\emptyset$, so $n=5$.  Then  $G\cong K_{3,2}$, which is also a contradiction as $\rho<\rho(F_1)$.

Suppose next that (b) holds.  From  \eqref{eq1}, we have
$|X_0'|+|X_1|+2|X_2'|+|Y_1|+|Y_2'|+|Y_2|\le 3-|X_0^*| \le 1$, so $X_1=X_2'=\emptyset$. It follows that $Y=Y_0$, say $Y_0=\{w\}$, so $G':=G+uw$  does not contain a DCC$_1$.  By Lemma \ref{addedges}, $\rho(G')>\rho$,  a contradiction.

Suppose finally that (c) holds. If $|X_0^*|\ge 2$, then as $Y_0=\emptyset$, each vertex in $X_0^*$ is a pendant vertex of $G$, so  adding an edge between two vertices in $X_0^*$ leads to a graph that does not contain a DCC$_1$, which,  by Lemma \ref{addedges}, is a contradiction. So $|X_0^*|=1$.

\begin{claim}\label{xempty}
$X_1\cup X_2=\emptyset$.
\end{claim}
\begin{proof}
Suppose to the contrary that $X_1\cup X_2\ne \emptyset$.

If $Y=\emptyset$, then adding an edge between the vertex of $X_0^*$ and a vertex in $X_1\cup X_2'$ produces a graph that does not contain a DCC$_1$, which, by Lemma \ref{addedges}, is a contradiction. So $Y\ne \emptyset$. Moreover, we have  $Y_2'=\emptyset$; otherwise we have by Claim \ref{cut} and  \eqref{eq1} that $|X_0'|=|X_2'|=|Y_2'|=1$, $X=X_0^*\cup X_0'\cup X_2\cup X_3$ and $Y=Y_2'$, so $e(X,Y)=2$, which contradicts \eqref{eqqq}. This shows that $Y_1\cup Y_2\ne \emptyset$.

Suppose that $Y_1\ne\emptyset$. If $X_0'=\emptyset$, then  the graph obtained from $G$ by adding an edge between the vertex of $X_0^*$ and a vertex in $Y_1$ does not contain a DCC$_1$, which,  by Lemma \ref{addedges}, is a contradiction. So $X_0'\ne\emptyset$. From \eqref{eq1}, $|Y_1|=1$, so the graph obtained from $G$ adding an edge between the vertex of $X_0^*$ and the vertex of $Y_1$ does not contain a DCC$_1$, which,  by Lemma \ref{addedges}, is a contradiction.

Suppose that $Y_2\ne \emptyset$. Then $X_2'\ne \emptyset$. We have by Eq.~\eqref{eq1} that $|X_2'|=1$, $|Y_2|=1,2$ and $e(X_2,Y_2)=2|Y_2|$. Let $F_0=G[\{u\}\cup X_0^*\cup X_2\cup Y_2]$. Then $F_0$ is obtained from $K_{2,2+|Y_2|}$ with bipartition $(X_2^*,\{u\}\cup X_2'\cup Y_2)$ by adding an edge between $u$ and the vertex of $X_2'$ and a pendant vertex to $u$. If $n=5+|Y_2|$, then $G=F_0$ and if $n>5+|Y_2|$, then $G$ consists of $F_0$ and $G[V(G)\setminus V(F_0)\cup \{u\}]$ with a common vertex $u$. Let $F_0'\cong K_{1,1,|Y_2|+3}$. By a direct calculation, we have
\[
\rho(F_0)=\begin{cases}
2.9439&\mbox{ if }|Y_2|=1,\\
3.2054&\mbox{ if }|Y_2|=2,
\end{cases}
\mbox{ and }
\rho(F_0-u)=\begin{cases}
2&\mbox{ if }|Y_2|=1,\\
2.4495&\mbox{ if }|Y_2|=2,
\end{cases}
\]
From Lemma \ref{comm},
\[
\rho(F_0')=\begin{cases}
3.3723&\mbox{ if }|Y_2|=1,\\
3.7016&\mbox{ if }|Y_2|=2.
\end{cases}
\]
As $F_0'-u\cong K_{1,|Y_2|+3}$, \[
\rho(F_0'-u)=\begin{cases}
2&\mbox{ if }|Y_2|=1,\\
\sqrt{5}&\mbox{ if }|Y_2|=2.
\end{cases}
\]
Thus,
\[
\rho(F_0)<\rho(F_0')\mbox{ and }\rho(F_0-u)\ge \rho(F_0'-u).
\]
If $n=5+|Y_2|$, then $\rho(F_0)<\rho(F_0')$, a contradiction, and if $n>5+|Y_2|$, then we have by Lemma \ref{comm2} that $\rho(F_0'uG[V(G)\setminus V(F_0)\cup \{u\}])>\rho$, also a contradiction.

This shows that $X_1\cup X_2=\emptyset$, as desired.
\end{proof}

By Claim \ref{xempty}, we have $X=X_0^*\cup X_3$ and  $Y=\emptyset$, so $G\cong K_1\vee (K_1\cup \tfrac{n-2}{3}K_3)$, which, by Lemma \ref{cal}, is a contradiction if $n\ge 8$.  Thus $n=5$ and $G\cong F_1$.

\noindent
{\bf Case 3.} $X_0^*=\emptyset$.

\noindent
{\bf Case 3.1.} $X_0'=\emptyset$.

In this case, $X_0=\emptyset$.
If $X_1\cup X_2=\emptyset$, then $G\cong K_1\vee \tfrac{n-3}{3}K_3$, which,  by Lemma \ref{cal}, is a contradiction. So $X_1\cup X_2\ne \emptyset$.

Next, we show that $Y_1\cup Y_2\cup Y_2'=\emptyset$. Suppose to the contrary that $Y_1\cup Y_2\cup Y_2'\ne \emptyset$.

\begin{claim}\label{yempty}
For each $w\in Y_1\cup Y_2\cup Y_2'$ and one neighbor $v$ of $w$ in $X_1\cup X_2$, any path from $w$ to $u$ passes through some vertex of the component $H$ of $G[X]$ containing $v$.
\end{claim}
\begin{proof}
Suppose that this is not true. let $P$ be a path
from $w$ to $u$ that does not pass through any vertex of $H$. Let $v_1$ be a neighbor of $u$ on $P$. As $X_0=\emptyset$, $N_X(v_1)\ne \emptyset$. Let $v_1'\in N_X(v_1)$. Assume that $v_1'\notin V(P)$. Then with $v'\in N_X(v)$, $uv'vwP[w,v_1]v_1v_1'u$ is a cycle with chords $uv$ and $uv_1$, a contradiction.
\end{proof}

\begin{claim}\label{b1}
$Y_2'=\emptyset$.
\end{claim}

\begin{proof}
Otherwise, assume that $w\in Y_2'$. Let $v$ be the neighbor of $w$ in $X_2'$. By Claims \ref{cut} and \ref{yempty}, there is a path from $w$ to $u$ containing some vertex $v^*\in N_X(v)$. Then,   with $v'\in N_X(v)\setminus\{v^*\}$, $uv'vwP[w,v^*]v^*u$ is a cycle with chords $uv$ and $vv^*$, a contradiction.
\end{proof}

By Claim \ref{b1}, we assume that $w\in Y_1\cup Y_2$. Let $v_1$ be one neighbor of $w$ in $X_1\cup X_2$ and $H$ be the component of $G[X]$ containing $v_1$, which is a copy of $K_2$ or $K_{1,2}$.
Denote by $v_2$ the pendant vertex in $H$ different from $v_1$.

Let $Y_H=N_Y(v_1)\cup N_Y(v_2)$.

\begin{claim}\label{OO}
For any $z\in Y_H$, $d_G(z)=2$.
\end{claim}

\begin{proof}
By Claims \ref{x20} and \ref{x1}, $z$ has one or two neighbors in $H$.

Suppose first that $z$ has two neighbors in $H$, that is, the two neighbors are $v_1$ and $v_2$. If $N_Y(z)\ne \emptyset$, say $z'\in N_Y(z)$, then by Claim \ref{cut}, there is a path from $z'$ to $u$ which does not pass through $z$. By Claims \ref{yempty} and \ref{b1}, a shortest such path $P$ passes through one of $v_1$ and $v_2$, say $v_1$. Then $v_2$ does not lie on $P$. So $uv_2zz'P[z',v_1]v_1u$ is a cycle with chords $v_1v_2$ and $v_1z$ if $H\cong K_2$ and $uv_0v_2zz'P[z,v_1]v_1u$ is a cycle with chords $v_0v_1$ and $v_1z$ if $H\cong K_{1,2}$ with center $v_0$, a contradiction.  It follows that $N_Y(z)=\emptyset$, so $d_G(z)=2$.

Suppose next that $z$ has exactly one neighbor, say $v_1$, in $H$.
Suppose that $d_Y(z)\ge 2$. Next, we show that $N_Y(z)\subseteq Y_H$. Suppose that this is not true. Then
$z$ has a neighbor  in $Y_0$, so $|Y_0|\ge 1$. Note that $X_0=\emptyset$. By Claim  \ref{x2}, $e(X^*_2,Y_2)\le 2|Y_2|$. By
Claim \ref{x1}, $e(X_1,Y_1)\le 2|Y_1|$. Now \eqref{eqqq} becomes
\[
3(|X_1|+|X_2|+|Y_1|+|Y_2|+|Y_0|)-6\le |X_1|+|X_2|+2|X_2|-2|X_2'|+2|Y_1|+2|Y_2|
\]
i.e., $|X_1|+2|X_2'|+3|Y_0|+|Y_1|+|Y_2|\le 6$, so $|Y_0|=1$. If $X_1\ne\emptyset$, then $z\in Y_H\subseteq Y_1$, so $Y_H=\{z\}$. If $X_2'\ne\emptyset$, then $z\in Y_H\subset Y_2$, so $Y_H=\{z\}$. In either case,  $z$ is a cut vertex of $G$, a contradicting Claim \ref{cut}. So $N_Y(z)\subseteq Y_H$. By Claims \ref{cut} and \ref{yempty}, there is a path $P$ from $z$ to $u$ containing $v_2$ but not $v_1$. Let $z'\in N_Y(z)$ outside $P$.  Assume that $z'v_1\in E(G)$. Then $uv_1z'zP[z,v_2]v_2u$ is a cycle with chords $v_1v_2$ and $v_1z$ if $H\cong K_2$ and $uv_1z'zP[z,v_2]v_2v_0u$ is a cycle with chords $v_0v_1$ and $v_1z$ if $H\cong K_{1,2}$  with center $v_0$, a contradiction. So $d_Y(z)\le 1$. By Claim \ref{cut}, $d_Y(z)=1$,  so $d_G(z)=2$.
\end{proof}

By Claims \ref{x2}, \ref{x1} and \ref{OO}, $e(Y_H)\le \tfrac{|Y_H|}{2}$, so
$\sum_{z\in Y_H}x_ux_z>\sum_{zz'\in E(G[Y_H])}x_zx_{z'}$.
Assume that $x_{v_1}\ge x_{v_2}$.

Suppose that $H\cong K_2$. Let
\begin{align*}
G'&=G-\{v_2z:z\in N_Y(v_2)\setminus N_Y(v_1)\}-\{zz'\in E(G):z,z'\in Y_H\}\\
&\quad +\{v_1z:z\in N_Y(v_2)\setminus N_Y(v_1)\}+\{uz:z\in Y_H \}.
\end{align*}
It is easy to see that  $G'$ does not contain a DCC$_1$. But we have by Rayleigh's principle that
\[
\rho(G')\ge \rho+
2\sum_{z\in N_Y(v_2)\setminus N_Y(v_1)}(x_{v_1}-x_{v_2})x_z+2\left(\sum_{z\in Y_H}x_ux_z-\sum_{zz'\in E(G[Y_H])}x_zx_{z'}\right)>\rho,
\]
a contradiction.

Suppose that $H\cong K_{1,2}$. Let $v_0$ be the center of $H$.
Let
\[
G'=
\begin{cases}
\begin{aligned}
&G-\{v_iz_i:z_i\in N_Y(v_i),i=1,2\}-\{zz'\in E(G):z,z'\in Y_H\}\\
&\quad +\{uz,v_0z:z\in Y_H\}
\end{aligned}
&
 \mbox{if $x_{v_0}\ge x_{v_1}$}, \\
 \begin{aligned}
&G-v_0v_2-\{v_2z:z\in N_Y(v_2)\}-\{zz'\in E(G):z,z'\in Y_H\}\\
 &\quad +v_1v_2+\{uz:z\in Y_H \}+\{v_1z:z\in N_Y(v_2)\setminus N_Y(v_1)\}
 \end{aligned}
& \mbox{otherwise}.
 \end{cases}
 \]
Note that $G'$ does not contain a DCC$_1$. Let $U_0=N_Y(v_1)\cap N_Y(v_2)$, $U_1=N_Y(v_1)\setminus U_0$ and $U_2=N_Y(v_2)\setminus U_0$. If   $x_{v_0}\ge x_{v_1}$,  as one of $U_0, U_1, U_2$ is not empty, so
\begin{align*}
\rho(G')& \ge \rho(G) +2\sum_{z\in U_0}(x_u+x_{v_0}-x_{v_1}-x_{v_2})x_z\\
&\quad +2\sum_{z\in U_1}(x_{v_0}-x_{v_1})x_z+2\sum_{z\in U_2}(x_{v_0}-x_{v_2})x_z\\
&\quad +2\left(\sum_{z\in U_1\cup U_2}x_ux_z-
\sum_{zz'\in E(G[U_1\cup U_2])}x_zx_{z'}\right)\\
&>\rho,
\end{align*}
and otherwise,   we have $x_{v_1}>x_{v_0}$, so
\begin{align*}
\rho(G') & \ge \rho+2(x_{v_1}-x_{v_0})x_{v_2}+2\sum_{z\in U_2}(x_{v_1}-x_{v_2})x_z\\
&\quad +2\sum_{z\in U_0}(x_u-x_{v_2})x_z+2\left(\sum_{z\in U_1\cup U_2}x_ux_z-\sum_{zz'\in E(G[U_1\cup U_2])}x_zx_{z'}\right)\\
&>\rho.
\end{align*}
Thus we reach a contradiction in either case.
This shows that $Y_1\cup Y_2\cup Y_2'=\emptyset$.
Thus, $Y=\emptyset$.

Suppose that there are two components in $G[X_1\cup X_2]$, say $H_1$ and $H_2$. Let $v_i$ be the center of $H_i$ for $i=1,2$ (if $H_i\cong K_2$, then $v_i$ is arbitrary). Assume that $x_{v_1}\ge x_{v_2}$. Let
\[
G'=G-\{v_2z: z\in N_X(v_2)\}\cup \{v_1z: z\in N_X(v_1)\}.
\]
Obviously, $G'$ does not contain a DCC$_1$. By Lemma \ref{perron}, $\rho(G')>\rho$, a contradiction. So $G[X_1\cup X_2]$ is connected, i.e., either $X_1\ne \emptyset$ or $X_2\ne \emptyset$.

If $X_1\ne \emptyset$, then $|X_1|=2$, so $G\cong K_1\vee (K_2\cup \tfrac{n-3}{3}K_3)$, which, by Lemma \ref{cal}, is a contradiction.

If $X_2\ne\emptyset$, then $|X_2'|=1$, $G\cong K_1\vee(K_{1,t}\cup \tfrac{n-t-2}{3}K_3)$ for some $2\le t\le n-2$,  so we have by Lemma \ref{cal2} that $t=n-2$, i.e., $G\cong K_{1,1,n-2}$.

\noindent
{\bf Case 3.2.} $X_0'\ne \emptyset$.

In this case, $X_1\cup X_2'\ne \emptyset$.
As $X_0'\ne \emptyset$,  we have   $v_0w\in E(G)$ for some $v_0\in X_0'$ and some $w\in Y_1\cup Y_2'$.

\noindent
{\bf Case 3.2.1.} $w\in Y_1$.

From \eqref{eq1},  $Y_2=Y_2'=\emptyset$.

Let $v_1$ be a neighbor of $w$ in $X_1$ and $v_2$ the neighbor of $v_1$ in $X$. Let $X_1'=\{v_1,v_2\}$, $X_0''=\cup_{v\in X_1'}X_v$, $Y^*=N_Y(v_1)\cup N_Y(v_2)$ and $Y'=N_Y(v_1)\cap N_Y(v_2)$.

By similar argument as in the proof of Claim \ref{eqx1}, $Y'=\{w\}$ or $Y'=\emptyset$.

Suppose that $Y'=\{w\}$. Then $e(X_0'',Y_1)=|X_0''|$.
As $e(X_1',Y_1)=|Y^*|+1$, we have by  \eqref{eqqq} that $|Y^*|=1$, $e(X_1',Y^*)=2$ and $1\le |X_0''|\le 3$. So $N_G(w)=X_1'\cup X_0''$.
Let $F_2=G[\{u\}\cup N_G[w]]$. Then $F_2$ is the graph obtained from $K_{2,|X_0''|+2}$ with bipartition $(\{u,w\}, N_G(w))$ by adding an edge $v_1v_2$.
Thus, if $n=4+|X_0''|$, then $G=F_2$, and if $n>4+|X_0''|$,  then $G$ consists of $F_2$ and $G[V(G)\setminus N_G[w]]$ with a common vertex $u$. Let $F_2'$ be a copy of $K_{1,1,|X_0''|+2}$ obtained from
$K_{2,|X_0''|+2}$ as above by adding an edge $uw$.
By a direct calculation, we have
\[
\rho(F_2)=\begin{cases}
2.8858&\mbox{ if }|X_0''|=1,\\
3.1413&\mbox{ if }|X_0''|=2,\\
3.4142&\mbox{ if }|X_0''|=3,\end{cases} \mbox { and }
\rho(F_2-u)=\begin{cases}
2.1701&\mbox{ if }|X_0''|=1,\\
2.3429&\mbox{ if }|X_0''|=2,\\
2.5141&\mbox{ if }|X_0''|=3.
\end{cases}
\]
From Lemma \ref{comp},
\[
\rho(F_2')=\begin{cases}
3&\mbox{ if }|X_0''|=1,\\
3.3723&\mbox{ if }|X_0''|=2,\\
3.7016&\mbox{ if }|X_0''|=3.
\end{cases}
\]
Evidently, $F_2'-u\cong K_{1, |X_0'|+2}$, so
\[
\rho(F_2'-u)=\begin{cases}
\sqrt{3}&\mbox{ if }|X_0''|=1,\\
2&\mbox{ if }|X_0''|=2,\\
\sqrt{5}&\mbox{ if }|X_0''|=3.
\end{cases}
\]
Thus
\[
\rho(F_2)<\rho(F_2') \mbox{ and } \rho(F_2-u)>\rho(F_2'-u).
\]
If $n=4+|X_0''|$, then $\rho(F_2)<\rho(F_2')$, a contradiction, and if $n>4+|X_0''|$, then
we have by Lemma \ref{comm2} that $\rho(F_2'uG[V(G)\setminus N_G[w]])>\rho$, also a contradiction.

Suppose that $Y'=\emptyset$. Then $e(X_1,Y_1)=|Y^*|$ and $e(X_0'',Y_1)\le |X_0''|+|Y^*|-1$, so we have by \eqref{eqqq} that $|Y^*|=1$ and $|X_0''|\le 2$. As $G+v_2w$ also does not contain a DCC$_1$, we have by Lemma \ref{addedges} that $\rho(G+v_2w)>\rho$, a contradiction.

\noindent
{\bf Case 3.2.2.} $w\in Y_2'$.

Let $v\in N_X(w)$. As $G$ does not contain a DCC$_1$, $N_Y(w)=\emptyset$.
So $G'=G-\{vz:z\in N_Y(v)\}+\{uz:z\in N_Y(v)\}$ also does not contain a DDC$_1$. By Lemma \ref{perron}, $\rho(G')>\rho$, a contradiction.

Combining the above cases, we have $G\cong F_1, K_{1,1, n-2}, K_{3,n-3}$.
Note that $\rho(F_1)>\rho(K_{1,1,3})>\rho(K_{3,2})$. By Lemma \ref{comp}, we have \eqref{FFF}.
\end{proof}

\section{Concluding remarks}

By Theorem \ref{double2},
if $G$ is a $5$-vertex graph containing no copies of $K_1\vee P_4$, then
$\rho(G)\le \rho(F_1)$ with equality if and only if $G\cong F_1$. Questions \ref{LY} and \ref{LY2} are different from the question to find spectral conditions that imply a graph on $n$ vertices contains a copy of $K_1\vee P_4$.

\begin{theorem}\label{double1} \cite{ZP}
Suppose that $G$ is an $n$-vertex graph containing no copies of $K_1\vee P_4$ with $n\ge 6$. Then
\[
\rho(G)\le \begin{cases}
\frac{n+1}{2}&\mbox{ if }n\equiv 1\pmod 2\\
\frac{1+\sqrt{n^2+1}}{2}&\mbox{ if }n\equiv 0\pmod 4\\
\eta(n)&\mbox{ if }n\equiv 2\pmod 4
\end{cases}
\]
with equality if and only if $G\cong H_n$, where
\[
H_n=\begin{cases}
\frac{n+1}{2}K_1\vee \frac{n-1}{4}K_2&\mbox{ if }n\equiv 1\pmod 4,\\
\frac{n-1}{2}K_1\vee \frac{n+1}{4}K_2&\mbox{ if }n\equiv 3\pmod 4,\\
\frac{n}{2}K_1\vee \frac{n}{4}K_2&\mbox{ if }n\equiv 0\pmod 4,\\
\frac{n}{2}K_1\vee (\frac{n-2}{4}K_2\cup K_1)&\mbox{ if }n\equiv 2\pmod 4,
\end{cases}
\]
and  $\eta(n)$ is the largest root of $x^3-x^2-\tfrac{n^2}{4}x+\tfrac{n}{2}=0$.
\end{theorem}

\bigskip

\noindent
{\bf Declaration of competing interest}

There is no competing interest.
%The authors have no relevant financial or non-financial interests to disclose.

\bigskip

\noindent
{\bf Data availability}

No data was used for the research described in the article.

\bigskip
\noindent {\bf Acknowledgements}

%The authors  would like to thank the reviewers for constructive comments and suggestions.
This work was supported by National Natural Science Foundation of China (No.~12071158).

\end{document}